\documentclass[12pt, draft]{amsart}

\usepackage{amssymb}
\usepackage{url}
\usepackage{amscd}
\usepackage{amsmath,amssymb,amsfonts}
\usepackage{mathrsfs}
\usepackage{pifont}
\input xy
\xyoption{all}

\theoremstyle{plain}
\newtheorem{theorem}{Theorem}[section]

\newtheorem{lemma}[theorem]{Lemma}
\newtheorem{proposition}[theorem]{Proposition}
\newtheorem{definition}[theorem]{Definition}
\theoremstyle{remark}
\newtheorem{remark}[theorem]{Remark}

\begin{document}


\title{About the Cuntz Comparison for non-simple $\text{C}^*$-algebras}

\author{George Elliott and Kun Wang}



\begin{abstract}
We study the Cuntz semigroup for non-simple $\text{C}^*$-algebras in this paper. In particular, we use the extended Elliott invariant to characterize the Cuntz comparison for $\text{C}^*$-algebras  with the projection property which have only one ideal. 
\end{abstract}

\maketitle


\newcommand\sfrac[2]{{#1/#2}}
\newcommand\cont{\operatorname{cont}}
\newcommand\diff{\operatorname{diff}}

\section{Introduction}
Classification of $\text{C}^*$-algebras by using the so-called Elliott invariant has been shown to be very successful in many cases. 
  Recent examples due first to R\o rdam and later Toms have shown the
currently proposed invariants to be insufficient for the
classification of all simple, separable, and nuclear
$\mbox{C}^{*}$-algebras. There are simple, separable, and
nuclear $\mbox{C}^{*}$-algebras that can be distinguished by their
Cuntz semigroups but not by their Elliott Invariant. So the Cuntz semigroup has
become popular and important. 

In \cite{BPT}, Brown, Perera and Toms recover the Cuntz semigroup for a class of simple $\text{C}^*$-algebras by using the ingredient of Elliott invariant---the Murry-von Neumann semigroup and lower semi-continuous dimension functions. 
In this paper, we give a characterization of Cuntz comparability for a class of $\text{C}^*$-algebras with only one ideal by using the Murry-von Neumann semigroup and lower semi-continuous dimension functions on the $\text{C}^*$-algebra and on its ideal. 
We will give a example to see that though the extended valued traces can reflect the ideal structure of $\text{C}^*$-algebras, the traces on ideals still have their own meaning.

\section{Cuntz Comparability}
Let $A$ be a $\text{C}^*$-algebra, and let $\text{M}_n(A)$ denote the $n\times n$ matrices whose entries are elements of $A$. Let $\text{M}_{\infty}(A)$ denote the algebraic limit of the direct system $(\text{M}_n(A),\phi_n)$, where $\phi_n:\text{M}_n(A)\rightarrow \text{M}_{n+1}(A)$ is given by
$$a\longmapsto \begin{pmatrix}
   a & 0  \\
   0 & 0
  \end{pmatrix}.$$
  Let $\text{M}_{\infty}(A)_+$ (resp. $\text{M}_{n}(A)$) denote the positive elements in $\text{M}_{\infty}(A)$ (resp. $\text{M}_{n}(A)$). For positive elements $a$ and $b$ in $\text{M}_{\infty}(A)$, write $a\oplus b$ to denote the element $\begin{pmatrix}
   a & 0  \\
   0 & b
  \end{pmatrix},$ which is also positive in $\text{M}_{\infty}(A).$
\begin{definition}\label{Cuntz subeq} Given $a,b\in\text{M}_{\infty}(A)_+$, we say that a is Cuntz subequivalent  to b (written $a\precsim b$) if there is a sequence $(v_n)_{n=1}^{\infty}$ of elements of $\text{M}_{\infty}(A)$ such that $$\|v_nbv_n^*-a\|\xrightarrow[ ]{n\rightarrow\infty}0.$$
\end{definition}
We say that $a$ and $b$ are Cuntz equivalent (written $a\sim b$) if $a\precsim b$ and $b\precsim a$.
This relation is an equivalence relation, and write $\langle a\rangle$ for the equivalence class of $a$. The set $$\text{W}(A):=\text{M}_{\infty}(A)_+/\sim$$
becomes a positively ordered Abelian monoid when equipped with the operation $$\langle a \rangle+\langle b\rangle=\langle a\oplus b\rangle$$ and the partial order $$\langle a\rangle\leq \langle b\rangle\Longleftrightarrow a\precsim b.$$ In the sequel, we refer to this object as the Cuntz semigroup of $A$.

\begin{proposition}([\cite{MR1}, Proposition 2.4])\label{equi} Let $A$ be a $\text{C}^*$-algebra, and let $a,b$ be positive elements in $A$.
The following are equivalent:
\begin{item}
\item
(i) $a\precsim b$,
\item
(ii) for all $\varepsilon>0,$ $(a-\varepsilon)_+\precsim b,$
\item
(iii) for all $\varepsilon>0$, there exists $\delta>0$ such that $(a-\varepsilon)_+\precsim (b-\delta)_+,$
\item
(iv) for all $\varepsilon>0,$ there exists $\delta>0$ and $x$ in $A$ such that $$(a-\varepsilon)_+=x^*(b-\delta)_+x.$$
\end{item}
\end{proposition}

\begin{proposition}\label{disum}(\cite{PT}, Proposition 2.2) Let $A$ be a $C^*$-algebra. Let $a,~p\in {\rm{M}}_{\infty}(A)_+$ be such that $p$ is a projection and $p\precsim a$. Then there exists $b\in {\rm{M}}_{\infty}(A)_+$ such that $p\oplus b\sim a.$
\end{proposition}

\begin{definition} Let $A$ be a local $C^*$-algebra. A dimension function on $A$ is a mapping $D:A\rightarrow [0,\infty)$ such that:

(i) $\sup\{D(a)|a\in A\}=1$ (normalization).

(ii) If $a\bot b$ (i.e., $ab=ab^*=a^*b=a^*b^*=0$), then $D(a+b)=D(a)+D(b)$.

(iii) For all $a$, $D(a)=D(aa^*)=D(a^*a)=D(a^*)$.

(iv) If $0\leq a\leq b$, then $D(a)\leq D(b)$.

(v) If $a\precsim b$ (i.e., there exists $x_n,~y_n$ with $\{x_nby_n\}$ converging to $a$ in norm), then $D(a)\leq D(b)$.
\end{definition}

\begin{proposition}\label{MR}(\cite{MR2}, Corollary 4.7) Let $A$ be a $C^*$-algebra for which $W(A)$ is almost unperforated (in particular, $A$ could be a $\mathcal{Z}$-absorbing $C^*$-algebra). Let $a,~b$ be positive elements in $A$. Suppose that $a$ belongs to $\overline{AbA}$ and such that $d(a)<d(b)$ for every dimension function $d$ on $A$ with $d(b)=1$. Then $a\precsim b.$
\end{proposition}
\begin{lemma}\label{MRL} (see also in \cite{ERS}) Let $A$ be an exact $C^*$-algebra for which $W(A)$ is almost unperforated (in particular, $A$ could be a $\mathcal{Z}$-absorbing $C^*$-algebra). Let $a,~b$ be positive elements in $A$. Suppose that $a$ belongs to $\overline{AbA}$ and such that $d_{\tau}(a)<d_{\tau}(b)$ for every dimension function $d_{\tau}$ on $A$ with $d_{\tau}(b)=1$. Then $a\precsim b.$
\end{lemma}
\begin{proof}Suppose $d$ is any dimension function on $A$ with $d(b)=1$. Let $$\bar{d}(\langle x\rangle)=\lim\limits_{\varepsilon\rightarrow 0}d(\langle f_{\varepsilon}(x)\rangle),$$ where
$$
f_{\varepsilon}(t) = \left\{
  \begin{array}{l l}
     0, & \quad \text{$t\leq \varepsilon$}\\
     \frac{t-\varepsilon}{\varepsilon}, & \quad \text{$\varepsilon\leq t\leq 2\varepsilon$}\\
     1, & \quad {t\geq 2\varepsilon}.
   \end{array} \right.
$$
Then $\bar{d}$ is a lower semi-continuous dimension function on $A$. Therefore, $\bar{d}=d_{\tau}$ for some $\tau\in \text{T}A$. (This follows from Blackadar and Handelman, [\cite{BH}, Theorem II.2.2].)

If $\bar{d}(b)=0$, then $a\in \overline{AbA}$ implies $a\precsim b\otimes 1_k$ for some integer $k$. Therefore, $\bar{d}(a)=0$ since $0\leq\bar{d}(a)\leq k\bar{d}(b)=0$. Hence $d(\langle f_{\varepsilon}(a)\rangle)\leq \bar{d_{\tau}}(a)=0$ for any $\varepsilon>0$.

If $\bar{d}(b)\neq 0$, then let $l\triangleq {\bar{d}}/{\bar{d}(b)}$, which is a lower semi-continuous dimension function with $l(b)=1.$ Then by the assumption
$$d(\langle f_{\varepsilon}(a)\rangle)/{\bar{d}(b)}\leq l(\langle a\rangle)<l(\langle b\rangle)\leq d(\langle b\rangle)/{\bar{d}(b)}.$$ Hence
$$d(\langle f_{\varepsilon}(a)\rangle)\leq \bar{d}(\langle a\rangle)<\bar{d}(\langle b\rangle)\leq d(\langle b\rangle).$$

In either case, we have $$d(\langle f_{\varepsilon}(a)\rangle)<d(\langle b\rangle) ~~~~~~~~~\text { for  }\forall\varepsilon>0.$$ Since $a\in \overline{AbA}$, $f_{\varepsilon}(a)\in \overline{AbA}$. Therefore, by Proposition \ref{MR} $$f_{\varepsilon}(a)\precsim b~~~~~~~~~\text { for  }\forall\varepsilon>0.$$
Hence $a\precsim b.$

\end{proof}

\begin{definition}\label{A++} Let $A$ be a $C^*$-algebra with the ideal property. Let $A_{++}$ be the set of $A_+$ consisting of all positive elements which are not Cuntz equivalent to a non-zero projection in any quotient of $A$.
\end{definition}

\begin{lemma}\label{Projl} Let $A$ be an exact $C^*$-algebra for which $W(A)$ is almost unperforated (in particular, $A$ could be a $\mathcal{Z}$-absorbing $C^*$-algebra). Let $a\in A_{++}$ and $p$ be a projection in $A$. Then $p\precsim a$ if and only if $p\in\overline{AaA}$ and $d_{\tau}(p)<d_{\tau}(a)$ for each ${\tau}\in {\rm{T}}A$ with $d_{\tau}(a)=1$.
\end{lemma}
\begin{proof} The reverse direction immediately follows from Lemma \ref{MRL}. Now suppose $p\precsim a$, then by Proposition \ref{disum}, there exists a positive element $c$ with $p\oplus c\sim a.$ Considering the quotient $A/(c)$, where $(c)$ stands for the ideal generated by $c$, we have $a\sim p$ in $A/(c)$.
Therefore, $0=a=p$ in $A/(c)$. Hence $p\in (c)$.

If $d_{\tau}(c)=0$, then $d_{\tau}(p)=0$, which implies $d_{\tau}(a)=0.$ Therefore, if $d_{\tau}(a)\neq 0$, then $d_{\tau}(c)\neq 0$, hence $d_{\tau}(p)<d_{\tau}(a).$

\end{proof}

\section{}

\begin{lemma}
Let $A$ be a $\text{C}^*$-algebra such that $0\rightarrow I \xrightarrow[ ]{\iota}A \xrightarrow[ ]{\pi} A/I\rightarrow 0$ is a short
exact sequence. Let $q\in A$ satisfying $\pi(q)$ is a projection in $A/I$ but $q$ is not a projection in $A$. Suppose $a\in A_{++}$. Then $q\precsim a$ if and only if \\
 (1) $d_{\tau}(q)<d_{\tau}(a) \mbox{ for every }\tau\in \text{T}A \mbox{ with }\ker\tau=I,$ and \\
 (2) $d_{\tau}(q)\leq d_{\tau}(a)$ for every $\tau\in\text{T}A$ with $\ker(\tau)=\{0\}$.
\end{lemma}

\begin{proof}
Suppose $q\precsim a$. Then $d_{\tau}(q)\leq d_{\tau}(a)$ for all $\tau\in\text{T}A$.
Let $[x]=\pi(x)$ be the equivalent class of $x$ in $A/I$ for any $x\in A.$
Since $[q]\precsim [a]$ in $A/I$ and $[q]$ is a projection, by Proposition \ref{disum}, there exists $b\in M_{\infty}(A)$
such that $$[q]\oplus[b]\sim[a].$$
Since $a\in A_{++}$, $[a]$ is not a projection.
So $[b]\neq 0.$
If $\tau\in \text{T}A$ with $\ker(\tau)=I,$ then $\tau $ is a lower semi-continuous trace on $A/I.$
Therefore, $$d_{\tau}([q])+d_{\tau}([b])=d_{\tau}([a]).$$
Thus, $d_{\tau}(q)<d_{\tau}(a)$ for $\tau\in \text{T}A$ with $\ker\tau=I$.

Now suppose the condition (1) and (2) hold. 
If $\tau\in \text{T}A$ satisfying $\ker(\tau)=0$, then $\tau$ is a faithful trace on $A$.
Since $q$ is not a projection in $A$, $0\in\sigma(q)$ is not an isolated point.
For any $\varepsilon>0,$ there exists $\delta\in\sigma(q)$ such that
$$d_{\tau}((q-\varepsilon)_+)<d_{\tau}((q-\delta)_+).$$
Thus $$d_{\tau}((q-\varepsilon)_+)<d_{\tau}((q)\leq d_{\tau}(a)\mbox{ for all }\tau\in\text{T}A \mbox{ with }\ker\tau=\{0\}.$$
Therefore, $$d_{\tau}((q-\varepsilon)_+)<d_{\tau}(a)\mbox{ for all }\tau\in\text{T}A.$$
By Lemma \ref{MRL}, $$(q-\varepsilon)_+\precsim a\mbox{ for all } \varepsilon>0.$$
Thus, $q\precsim a.$

\end{proof}

\begin{remark}
For the second part of above proof, $a$ can be any positive element. 
That is, we do not require  $a$ that belong to $A_{++}$.
\end{remark}

\begin{lemma}
Let $A $ be an exact, $\mathcal{Z}$-stable $\text{C}^*$-algebra. Suppose $A$ has only one non-trivial ideal $I$, which is also exact and $\mathcal{Z}$-stable. Let $a\in A_{++}$ and $b\in A_{+}$. Then $a\precsim b$ if and only if
$d_{\tau}(a)\leq d_{\tau}(b)$ for every $\tau\in \text{T}A$, 
\end{lemma}
\begin{proof}
The sufficiency is clear.

If $a=0$, then the conclusion is true. So we assume $a\neq 0.$

If $a,b\in I$, then for any $\tau\in\text{T}I$, $\tau$ can be extended to an element in $\text{T}A$.
Thus, $d_{\tau}(a)\leq d_{\tau}(b)$ for every $\tau\in \text{T}A$ implies
$d_{\tau}(a)\leq d_{\tau}(b)$ for every $\tau\in \text{T}I$.
Since $I$ is simple, $a\precsim b$ follows from the Proposition 2.6 of \cite{PT}.

Since $a\in A_{++}$ and $a\neq0,$ there is a strictly decreasing sequence $\varepsilon_n$ of positive real numbers in $\sigma(a)$ converge to 0.

If $a\in I,~b\notin I$, then $d_{\tau}((a-\varepsilon_n)_+)\leq d_{\tau}(a)=0<d_{\tau}(b)$ for all $\tau\in\text{T}A$ with $\ker\tau=I$ and all $n\in\mathbb{N}$.
For $\tau\in\text{T}A$ with $\ker\tau=\{0\},$ we have
$$d_{\tau}((a-\varepsilon_n)_+)<d_{\tau}(a)\leq d_{\tau}(b) \mbox{ for all } n\in\mathbb{N}.$$
Therefore, in this case we can get
$$d_{\tau}((a-\varepsilon_n)_+)<d_{\tau}(b) \mbox{ for all }\tau\in \text{T}A,\mbox{ and all }n\in\mathbb{N}.$$
By Lemma \ref{MRL}, $(a-\varepsilon_n)_+\precsim b$ for all $n\in\mathbb{N}$.
Since the set $\{x\in A_+|x\precsim b\}$ is closed, and since $(a-\varepsilon_n)\rightarrow a$ in norm, we have $a\precsim b.$

If $a\notin I,$ and $b\notin I,$ for $\tau\in\text{T}A$ with $\ker\tau=\{0\},$ we have
$$d_{\tau}((a-\varepsilon_n)_+)<d_{\tau}(a)\leq d_{\tau}(b) \mbox{ for all } n\in\mathbb{N}.$$
For any $\tau\in\text{T}A$ with $\ker\tau=I$, then $\tau$ is a lower semi-continuous trace on $A/I$ and $d_{\tau}([a])\leq d_{\tau}([b]).$
Since $a\in A_{++}$ and $[a]\neq0,$ there is a strictly decreasing sequence $\delta_m$ of positive real numbers in $\sigma([a])$ converge to 0. Since $A/I$ is simple, every trace on $A/I$ is faithful.
$$d_{\tau}(([a]-\delta_n)_+)<d_{\tau}([a])\leq d_{\tau}([b])$$
  for all  $\tau\in \text{T}(A)$ with $\ker\tau=I$ { and all } $m\in\mathbb{N}.$
  Therefore, $$d_{\tau}((a-\delta_n)_+)< d_{\tau}(b)$$
  for all  $\tau\in \text{T}(A)$ with $\ker\tau=I$ { and all } $m\in\mathbb{N}.$
  Since both $\{\varepsilon_n\}$ and $\{\delta_m\}$ are sequences converging to 0, there exist subsequence $n_k$ and $m_k$
  such that $\varepsilon_{n_1}\geq\delta_{m_1}\geq\varepsilon_{n_2}\geq\delta_{m_2}\cdots.$
Without loss of generality, we can assume $n_k=m_k=k.$
Therefore, $$d_{\tau}((a-\varepsilon_n)_+)< d_{\tau}(b)\mbox{ for all }\tau\in\text{T}A \mbox{ and all }n\in\mathbb{N}.$$
Thus, $(a-\varepsilon_n)_+\precsim b$ for all $n\in\mathbb{N.}$ So $a\precsim b.$

\end{proof}

\section{  }

\begin{lemma}\label{dir}
Let $A$ be a $\text{C}^*$-algebra and $x,y\in A$ be two positive elements.
For any $\varepsilon>0$, there exists $\delta>0$ such that:
if $$\|x^{\frac{1}{2}}y^{\frac{1}{2}}\|<\min\{\delta,\frac{\delta}{\|x+y\|},1\},$$ then
$$[\begin{pmatrix}
   x & 0  \\
   0 & y
  \end{pmatrix}-\varepsilon]_+\precsim \begin{pmatrix}
   x+y & 0  \\
   0 & 0
  \end{pmatrix}.$$
\end{lemma}
\begin{proof}
Since $\begin{pmatrix}
   x^2 & 0  \\
   0 & y^2
  \end{pmatrix}\sim\begin{pmatrix}
   x & 0  \\
   0 & y
  \end{pmatrix},$
for any $\varepsilon>0$, there exists $\delta>0$ such that
$$[\begin{pmatrix}
   x & 0  \\
   0 & y
  \end{pmatrix}-\varepsilon]_+\precsim[\begin{pmatrix}
   x^2 & 0  \\
   0 & y^2
  \end{pmatrix}-\delta]_+.$$
  Let
    \begin{align*}
  \omega &=\begin{pmatrix}
   x^{\frac{1}{2}} & y^{\frac{1}{2}}  \\
   y^{\frac{1}{2}} & x^{\frac{1}{2}}
  \end{pmatrix}\begin{pmatrix}
   x+y & 0  \\
   0 & 0
  \end{pmatrix}\begin{pmatrix}
   x^{\frac{1}{2}} & y^{\frac{1}{2}}  \\
   y^{\frac{1}{2}} & x^{\frac{1}{2}}
  \end{pmatrix}\\
  &=\begin{pmatrix}
   x^2+x^{\frac{1}{2}}yx^{\frac{1}{2}} & x^{\frac{3}{2}}y^{\frac{1}{2}}+x^{\frac{1}{2}}y^{\frac{3}{2}}  \\
   y^{\frac{1}{2}}x^{\frac{3}{2}}+y^{\frac{3}{2}}x^{\frac{1}{2}} & y^{\frac{1}{2}}xy^{\frac{1}{2}}+y^2
  \end{pmatrix}.
\end{align*}
Then 
\begin{align*}
&\|w-\begin{pmatrix}
x^2 & 0\\
0 & y^2
\end{pmatrix}\| \\
=&\| \begin{pmatrix}
x^{\frac{1}{2}}yx^{\frac{1}{2}} & x^{\frac{3}{2}}y^{\frac{1}{2}}+x^{\frac{1}{2}}y^{\frac{3}{2}}\\
y^{\frac{1}{2}}x^{\frac{3}{2}}+y^{\frac{3}{2}}x^{\frac{1}{2}} & y^{\frac{1}{2}}xy^{\frac{1}{2}}
\end{pmatrix}
\|<\delta,
\end{align*}
since $$\|x^{\frac{1}{2}}yx^{\frac{1}{2}}\|=\|(x^{\frac{1}{2}}y^{\frac{1}{2}})(x^{\frac{1}{2}}y^{\frac{1}{2}})^*\|=\|x^{\frac{1}{2}}y^{\frac{1}{2}}\|^2<\delta,$$
\begin{align*}
&\|x^{\frac{3}{2}}y^{\frac{1}{2}}+x^{\frac{1}{2}}y^{\frac{3}{2}} \|=\|x^{\frac{1}{2}}(x+y)y^{\frac{1}{2}} \| 
=\|(x+y)^{\frac{1}{2}}y^{\frac{1}{2}}x^{\frac{1}{2}}(x+y)^{\frac{1}{2}} \| \\
\leq&\|(x+y)^{\frac{1}{2}}\|^2\|x^{\frac{1}{2}}y^{\frac{1}{2}}\|=\|x+y\|\|x^{\frac{1}{2}}y^{\frac{1}{2}}\|<\delta.
\end{align*}
Therefore, $$[\begin{pmatrix}
   x & 0  \\
   0 & y
  \end{pmatrix}-\varepsilon]_+\precsim[\begin{pmatrix}
   x^2 & 0  \\
   0 & y^2
  \end{pmatrix}-\delta]_+\precsim \omega\precsim \begin{pmatrix}
   x+y & 0  \\
   0 & 0
  \end{pmatrix}.$$

\end{proof}

\begin{lemma}\label{fi}
Let $A$ be a $\text{C}^*$-algebra such that
$$0\rightarrow I \xrightarrow[ ]{\iota}A \xrightarrow[ ]{\pi} A/I\rightarrow 0$$
is a short exact sequence, where $I $ is a closed two-sided ideal of $A$.
Suppose $\{p_n\}$ is a quasi-central approximate unit of $I$ consisting of projections.
Let $a,b$ be two positive elements in $A.$
Then $[\pi(a)]\precsim [\pi(b)] \mbox{ in } A/I$ if and only if, for any $\varepsilon>0,$ there exists an integer $N>0$ such that
$$[(1-p_n)a(1-p_n)-\varepsilon]_+\precsim (1-p_n)b(1-p_n) $$
for all $n\geq N$.
\end{lemma}
\begin{proof}
Necessity. If for any $\varepsilon>0,$ there exists an integer $N>0$ such that
$$[(1-p_n)a(1-p_n)-\varepsilon]_+\precsim (1-p_n)b(1-p_n) $$
for all $n\geq N$.
Then
\begin{align*}
(\pi(a)-\varepsilon)_+&=(\pi((1-p_N)a(1-p_N))-\varepsilon)_+\\
&=\pi((1-p_N)a(1-p_N)-\varepsilon)_+\\
&\precsim \pi((1-p_N)b(1-p_N))\\
&=\pi(b).
\end{align*}
Since the above relation holds for all $\varepsilon>0$, by Proposition \ref{equi}, $\pi(a)\precsim \pi(b)$.

Sufficiency. if $[\pi(a)]\precsim [\pi(b)] \mbox{ in } A/I$, then there exist $\{v_k\}_{k=1}^{\infty}$ in $A$,
such that $$[\pi(a)]=\lim\limits_{k\rightarrow \infty}[\pi(v_k)][\pi(b)][\pi(v_k)]^*.$$
Given $\varepsilon>0,$ we can find  $k_0$ such that $$\|[\pi(a)]-[\pi(v_{k_0})][\pi(b)][\pi(v_{k_0})]^*\|<\varepsilon/2.$$
Thus there exists $d\in A$ such that $$([\pi(a)]-\varepsilon/2)_+=[\pi(d)][\pi(b)][\pi(d)]^*.$$
Since $[\pi(a)]=[\pi((1-p_n)a(1-p_n))]$ and $[\pi(b)]=[\pi((1-p_n)b(1-p_n))]$ for all $n\in\mathbb{N}$, we can find $c_n\in I$ such that
$$((1-p_n)a(1-p_n)-\varepsilon/2)_+=d(1-p_n)b(1-p_n)d^*+c_n.$$
Multiplying $p_n$ from both left and right sides of the above equation, we get
$$p_n((1-p_n)a(1-p_n)-\varepsilon/2)_+p_n=p_nd(1-p_n)b(1-p_n)d^*p_n+p_nc_np_n.$$
Since $p_n$ is a quasi central approximate unit of $I$,
$$\lim\limits_n\|p_n((1-p_n)a(1-p_n)-\varepsilon/2)_+p_n\|=0,$$
$$\lim\limits_n\|p_nd(1-p_n)b(1-p_n)d^*p_n\|=0.$$
Therefore, $\lim\limits_n \|p_nc_np_n\|=0.$ Thus,
$\lim\limits_n \|c_n\|=0.$
Then we can find a natural number $N$ such that for all $n\geq N,$
$$\|((1-p_n)a(1-p_n)-\varepsilon/2)_+-d(1-p_n)b(1-p_n)d^*\|<\varepsilon/2.$$
Thus $$(((1-p_n)a(1-p_n)-\varepsilon/2)_+-\varepsilon/2)_+\precsim d(1-p_n)b(1-p_n)d^*$$
 for all $n\geq N.$ That is,
 $$((1-p_n)a(1-p_n)-\varepsilon)_+\precsim d(1-p_n)b(1-p_n)d^*\mbox{ for all }n\geq N.$$

\end{proof}

\begin{lemma}\label{exa}
Let $A$ be a separable $\text{C}^*$-algebra and $I$ be an ideal of $A$.
Suppose $\{p_n\}_{n=1}^{\infty}$ is a quasi-central approximate unit of $I$.
If $a,b$ are two positive elements in $A$ satisfying
$$\pi(a)\precsim \pi(b) \mbox{ and }p_nap_n\precsim p_nbp_n$$  for all sufficiently large $n.$
Then $a\precsim b$.
\end{lemma}
\begin{proof}
By assumption, we know that
\begin{align*}
a&=(a-p_nap_n)+p_nap_n\\
&\precsim \begin{pmatrix}
   a-p_nap_n & 0  \\
   0 & p_nap_n
 \end{pmatrix}\\
 &\precsim \begin{pmatrix}
   b-p_nbp_n & 0  \\
   0 & p_nbp_n
  \end{pmatrix}
  \end{align*}
  For any $\varepsilon>0$, there exists $\varepsilon'>0$ such that
  $$(a-\varepsilon)_+\precsim [\begin{pmatrix}
   b-p_nbp_n & 0  \\
   0 & p_nbp_n
  \end{pmatrix}-\varepsilon']_+.\quad (*)$$
  Since $\{p_n\}$ is a quasi-central  approximate unit for $I$,
  $$(b-p_nbp_n)p_nbp_n=(1-p_n)bp_nbp_n\rightarrow 0.$$
  For $\varepsilon'>0,$ applying Lemma \ref{dir}, there exists $\delta$ and $N\in\mathbb{N}$ such that: for all $n\geq N$ we have
$$\|(b-p_nbp_n)p_nbp_n\|<\min\{\delta,\frac{\delta}{\|b\|},1\}.$$
Therefore, 
$$[\begin{pmatrix}
  b-p_nbp_n & 0  \\
 0 & p_nbp_n
  \end{pmatrix}-\varepsilon']_+\precsim \begin{pmatrix}
   b & 0  \\
   0 & 0
  \end{pmatrix}.$$
  for all $n\geq N.$
  Combining with $(*)$, we get $$(a-\varepsilon)_+\precsim b\mbox{ for all }\varepsilon>0.$$
  Thus, $a\precsim b.$

\end{proof}

\section{ }
For a $\text{C}^*$-algebra $A$ with one ideal $I$ which has the projection property, we define  
$$\widetilde{W}(A)=V(A)\sqcup LAff(T(A))^{++}\sqcup LAff(T(I))^{++}$$ and a map 
$$\iota: W(A)\rightarrow \widetilde{W}(A)$$ by 
$$\iota(a)=\left\{
\begin{array}{ll}
[a] & \text{ if } a\in V(A)\\
d_{\cdot}(a) & \text{ if } a\in A_{++}\\
d_{\cdot}(a) & \text{ if } a\in \mathcal{P}(A/I)
\end{array}
\right.
$$

\begin{theorem}
$\iota$ is an isomorphism.
\end{theorem}


\end{document}